\documentclass[10pt]{amsart}

\usepackage{amssymb}
\usepackage{amsthm}
\usepackage{amsmath}

\usepackage[usenames]{color}
\definecolor{ANDREW}{RGB}{255,127,0}

\usepackage{tikz}
\usetikzlibrary{arrows,calc}

\usepackage[foot]{amsaddr}

\usepackage{hyperref}

\theoremstyle{plain}
\newtheorem{proposition}{Proposition}[section]
\newtheorem{theorem}[proposition]{Theorem}
\newtheorem{lemma}[proposition]{Lemma}

\theoremstyle{definition}

\theoremstyle{remark}

\DeclareMathOperator{\Cc}{\mathcal{C}}

\DeclareMathOperator{\Oc}{\mathcal{O}}

\DeclareMathOperator{\Cb}{\mathbb{C}}

\DeclareMathOperator{\Nb}{\mathbb{N}}

\DeclareMathOperator{\Rb}{\mathbb{R}}

\newcommand{\norm}[1]{\left\|#1\right\|}

\begin{document}

\title[Lower bound for distances]{A lower bound for the K\"ahler-Einstein distance from the  Diederich-Forn{\ae}ss index}
\author{Andrew Zimmer}\address{Department of Mathematics, Louisiana State University, Baton Rouge, LA, USA}
\email{amzimmer@lsu.edu}
\date{\today}
\keywords{}
\subjclass[2010]{}

\begin{abstract} In this note we establish a lower bound for the distance induced by the K\"ahler-Einstein metric on pseudoconvex domains with positive  hyperconvexity index (e.g. positive Diederich-Forn{\ae}ss index). A key step is proving an analog of the Hopf lemma for Riemannian manifolds with Ricci curvature bounded from below.
\end{abstract}

\maketitle

\section{Introduction}

Every bounded pseudoconvex domain $\Omega \subset \Cb^d$ has a unique complete K{\"a}hler-Einstein metric, denoted by $g_{KE}$, with Ricci curvature $-(2d-1)$. This was constructed by Cheng and Yau~\cite{CY1980} when $\Omega$ has $\Cc^2$ boundary and by Mok and Yau~\cite{MY1983} in general. 

Let $d_{KE}$ be the distance induced by $g_{KE}$. Since $g_{KE}$ is complete, if we fix $z_0 \in \Omega$, then 
\begin{align}
\label{eq:limit_to_infinity}
\lim_{z \rightarrow \partial \Omega} d_{KE}(z,z_0) =\infty.
\end{align}
In this note we consider quantitative versions of Equation~\eqref{eq:limit_to_infinity}. In particular, it is natural to ask for lower bounds on $d_{KE}(z,z_0)$ in terms of the distance to the boundary function 
\begin{align*}
\delta_\Omega (z) = \min \{ \norm{w-z} : w \in \partial \Omega \}.
\end{align*}
Mok and Yau proved for every $z_0 \in \Omega$ there exists $C_1, C_2  \in \Rb$ such that 
 \begin{align*}
 d_{KE}(z,z_0) \geq C_1 + C_2 \log   \log  \frac{1}{\delta_\Omega(z)} 
 \end{align*}
 for all $z \in \Omega$, see~\cite[pg.  47]{MY1983}. Further, by considering the case of a punctured disk, this lower bound is the best possible for general pseudoconvex domains. 
 
 However, for certain classes of bounded pseudoconvex domains, there are much better lower bounds. For instance, if $\Omega$ is  convex, then for any $z_0\in \Omega$ there exists $C_1, C_2 > 0$ such that 
 \begin{align}
 \label{eq:good_form}
 d_{KE}(z,z_0) \geq C_1 + C_2 \log  \frac{1}{\delta_\Omega(z)} 
 \end{align}
 for all $z \in \Omega$, see~\cite{F1991}. In this note, we show that Estimate~\eqref{eq:good_form} holds for a large class of domains - those with positive hyperconvexity index. 
 
First we recall the well studied Diederich-Forn{\ae}ss index. Suppose $\Omega \subset \Cb^d$ is a bounded pseudoconvex domain. A number $\tau \in (0,1)$ is a called an \emph{Diederich-Forn{\ae}ss exponent of $\Omega$} if there exist a continuous plurisubharmonic function $\psi :\Omega \rightarrow (-\infty,0)$ and a constant $C > 1$ such that 
\begin{align*}
\frac{1}{C} \delta_\Omega(z)^{\tau} \leq -\psi(z) \leq C \delta_\Omega(z)^{\tau}
\end{align*}
for all $z \in \Omega$. Then the \emph{Diederich-Forn{\ae}ss index of $\Omega$} is defined to be 
\begin{align*}
\eta(\Omega) : = \sup\{ \tau: \tau \text{ is a Diederich-Forn{\ae}ss exponent of } \Omega \}.
\end{align*}
It is known that $\eta(\Omega) > 0$ for many domains. For instance, Diederich-Forn{\ae}ss~\cite{DF1977} proved that $\eta(\Omega) > 0$ when $\partial \Omega$ is $\Cc^2$. Later, Harrington~\cite{H2008} generalized this result and proved that $\eta(\Omega) > 0$ when $\partial \Omega$ is Lipschitz. 

The hyperconvexity index, introduced by Chen~\cite{C2017}, is a similar quantity associated to a bounded pseudoconvex domain $\Omega \subset \Cb^d$. In particular, a number $\tau \in (0,1)$ is a called an \emph{hyperconvexity exponent of $\Omega$} if there exist a continuous plurisubharmonic function $\psi :\Omega \rightarrow (-\infty,0)$ and a constant $C > 1$ such that 
\begin{align*}
 -\psi(z) \leq C \delta_\Omega(z)^{\tau}
\end{align*}
for all $z \in \Omega$. Then the \emph{hyperconvex index of $\Omega$} is defined to be 
\begin{align*}
\alpha(\Omega) : = \sup\{ \tau: \tau \text{ is a hyperconvexity exponent of } \Omega \}.
\end{align*}
By definition $\alpha(\Omega) \geq \eta(\Omega)$. Further, it is sometimes easier to verify that the hyperconvexity index is positive (see~\cite[Appendix]{C2017}). 

For domains with positive hyperconvexity index we will establish the following lower bound for $d_{KE}$. 

\begin{theorem}\label{thm:KE}Suppose $\Omega \subset \Cb^d$ is a bounded pseudoconvex domain with $\alpha(\Omega) > 0$. If $z_0 \in \Omega$ and $\epsilon >0$, then there exists some $C=C(z_0, \epsilon) \leq0$ such that 
\begin{align*}
d_{KE}(z,z_0) \geq C + \left( \frac{\alpha(\Omega)}{2d-1}-\epsilon\right) \log \frac{1}{\delta_\Omega(z)}
\end{align*}
for all $z \in \Omega$. 
\end{theorem}

In this note we have normalized the K{\"a}hler-Einstein metric to have Ricci curvature equal to $-(2d-1)$. If we instead normalized so that the Ricci curvature equals $-(2d-1)\lambda$ we would obtain the lower bound
\begin{align*}
C + \frac{1}{\sqrt{\lambda}}\left( \frac{\alpha(\Omega)}{2d-1}-\epsilon\right) \log \frac{1}{\delta_\Omega(z)}.
\end{align*}
  
In fact, we will show that Estimate~\eqref{eq:good_form} holds for any complete K\"ahler metric with Ricci curvature bounded from below. 

\begin{theorem}\label{thm:main} Suppose $\Omega \subset \Cb^d$ is a bounded pseudoconvex domain with $\alpha(\Omega) > 0$, $g$ is a complete K{\"a}hler metric on $\Omega$ with ${ \rm Ric}_g \geq -(2d-1)$, and $d_g$ is the distance associated to $g$. If $z_0 \in \Omega$ and $\epsilon > 0$, then there exists some $C=C(z_0,\epsilon) \leq0$  such that 
\begin{align*}
d_{g}(z_0, z) \geq C+ \left( \frac{\alpha(\Omega)}{2d-1}-\epsilon\right) \log \frac{1}{\delta_\Omega(z)}
\end{align*}
for all $z \in \Omega$. 
\end{theorem}

\subsection{Lower bounds on the Bergman metric} It is conjectured that the Bergman distance on a bounded pseudoconvex domain with $\Cc^2$ boundary also satisfies Estimate~\eqref{eq:good_form}. In this direction, the best general result is due B{\l}ocki~\cite{B2005} who extended work of Diederich-Ohsawa~\cite{DO1995} and established a lower bound of the form 
 \begin{align*}
 C_1 + C_2\frac{ 1}{\log \log \left(   1/ \delta_\Omega(z) \right)} \log \frac{1}{\delta_\Omega(z)}
 \end{align*}
for the Bergman distance on a bounded pseudoconvex domain with $\Cc^2$ boundary. 

Notice that Theorem~\ref{thm:main} implies the conjectured lower bound for the Bergman distance under the additional assumption that the Ricci curvature of the Bergman metric is bounded from below.

 \subsection*{Acknowledgements} I would to thank Yuan Yuan and  Liyou Zhang for bringing the hyperconvexity index to my attention. This material is based upon work supported by the National Science Foundation under grant DMS-1904099.

\section{A Hopf Lemma for Riemannian manifolds}

The standard proof of the Hopf lemma implies the following estimate:

\begin{proposition}[Hopf Lemma]\label{prop:classical_hopf} If $D \subset \Rb^d$ is a bounded domain with $\Cc^2$ boundary and $\varphi : D \rightarrow (-\infty,0)$ is subharmonic, then there exists $C > 0$ such that 
\begin{align*}
\varphi(x) \leq -C\delta_D(x)
\end{align*}
for all $x \in D$. 
\end{proposition} 

%
%
%
 
We will prove a variant of (this version of) the Hopf Lemma for Riemannian manifolds with Ricci curvature bounded below. 

Given a complete Riemannian manifold $(X,g)$, let $d_g$ denote the distance induced by $g$, let $\nabla_g$ denote the gradient, and let $\Delta_g$ denote the Laplace-Beltrami operator on $X$. A function $\varphi : X \rightarrow \Rb$ is \emph{subharmonic} if $\Delta_g \varphi \geq 0$ in the sense of distributions.

\begin{proposition}\label{prop:hopf} Suppose that $(X,g)$ is a complete Riemannian manifold with ${ \rm Ric}(g) \geq -(2d-1)$. If $x_0 \in X$, $\epsilon > 0$, and  $\varphi : X \rightarrow (-\infty,0)$ is subharmonic, then there exists $C> 0$ such that
\begin{align*}
\varphi(x) \leq -C\exp \Big( -(2d-1+\epsilon)d_g(x,x_0) \Big)
\end{align*}
for all $x \in X$. 
\end{proposition}

We require one lemma. Given a complete Riemannian manifold $(X,g)$, $x \in X$, and $r > 0$ define
\begin{align*}
B_g(x,r) = \{ y \in X : d_g(x,y) < r\}.
\end{align*}

\begin{lemma} Suppose that $(X,g)$ is a complete Riemannian manifold with ${ \rm Ric}(g) \geq -(2d-1)$. Then for every $x_0 \in X$ and $\epsilon > 0$, there exists $r_0 > 0$ such that the function
\begin{align*}
\Phi(x)= \exp \Big( -(2d-1+\epsilon)d_g(x,x_0) \Big)
\end{align*}
is subharmonic on $X \setminus B_g(x_0,r_0)$.
\end{lemma}

When the function $x \rightarrow d_g(x,x_0)$ is smooth on $X \setminus \{x_0\}$, the lemma is an immediate consequence of the Laplacian comparison theorem. We prove the general case by simply modifying the proof of the Laplacian comparison theorem given in~\cite{P2016}. 

\begin{proof} Let $r(x) = d_g(x,x_0)$. We will show that 
\begin{align*}
\Delta_g \Phi(x) \geq \Phi(x) \left( (2d-1+\epsilon)^2 -(2d-1)(2d-1+\epsilon)  \coth r(x) \right)
\end{align*}
in the sense of distributions on $X \setminus \{x_0\}$, which implies the lemma. 

Fix $q \in X$ and let $\sigma : [0,T] \rightarrow X$ be a unit speed geodesic joining $x_0$ to $q$. Then for $\delta \in (0,T)$ consider the function $r_{q,\delta}(x) = d_g(x,\sigma(\delta)) + \delta$. By the proof of~\cite[Lemma 7.1.9]{P2016}, $q$ is not in the cut locus of $\sigma(\delta)$. In particular, there exists a neighborhood $\Oc_q$ of $q$ such that $r_{q,\delta}$ is $\Cc^\infty$ and   
\begin{align*}
\norm{\nabla_g r_{q,\delta}} \equiv 1
\end{align*}
on $\Oc_q$, see~\cite[Proposition III.4.8]{S1996}. Further, by the Laplacian comparison theorem 
\begin{align*}
\Delta_g r_{q,\delta}(x) \leq (2d-1)\coth \left(r_{q,\delta}(x)-\delta\right)
\end{align*}
on $\Oc_q$, see~\cite[Lemma 7.1.9]{P2016}. Next consider the function $\Phi_{q,\delta} : \Oc_q \rightarrow [0,\infty)$ defined by  
\begin{align*}
\Phi_{q,\delta}(x) = \exp \Big( -(2d-1+\epsilon)r_{q,\delta}(x) \Big).
\end{align*}
Then 
\begin{align}\label{eq:smooth_Lap}
\Delta_g & \Phi_{q,\delta}(x) = \Phi_{q,\delta}(x) \left( (2d-1+\epsilon)^2\norm{\nabla_g r_{q,\delta}}^2  -(2d-1+\epsilon)\Delta_g r_{q,\delta}(x)\right)\nonumber\\
&\geq \Phi_{q,\delta}(x) \left( (2d-1+\epsilon)^2  -(2d-1+\epsilon)(2d-1)\coth \left(r_{q,\delta}(x)-\delta\right)\right).
\end{align}

Fix a partition of unit $1 = \sum_{j=1}^\infty \chi_j$ subordinate to the open cover $X = \cup_{q \in X} \Oc_q$. For each $j \in \Nb$, fix $q_j \in X$ such that ${\rm supp}(\chi_j) \subset \Oc_{q_j}$. 

Now suppose that $\psi : X \setminus \{x_0\} \rightarrow [0,\infty)$ is a compactly supported smooth function. Then by the dominated convergence theorem (notice that the sum is finite)
\begin{align*}
\int_X \Phi(x) \Delta_g \psi(x) dx = \lim_{\delta \rightarrow 0^+} \sum_{j=1}^{\infty} \int_{\Oc_{q_j}} \Phi_{q_j,\delta}(x) \Delta_g(\chi_j(x)\psi(x)) dx.
\end{align*}
By integration by parts and Equation~\eqref{eq:smooth_Lap}
\begin{align*}
\int_{\Oc_{q_j}} & \Phi_{q_j,\delta}(x) \Delta_g(\chi_j(x)\psi(x)) dx=  \int_{\Oc_{q_j}} \chi_j(x)\psi(x)\Delta_g\Phi_{q_j,\delta}(x)  dx \\
& \geq  \int_{\Oc_{q_j}} \chi_j(x)\psi(x)\Phi_{q,\delta}(x) \left( (2d-1+\epsilon)^2  -(2d-1+\epsilon)(2d-1)\coth \left(r_{q,\delta}(x)-\delta\right)\right)dx.
\end{align*}
So by applying the dominated convergence theorem again
\begin{align*}
\int_X \Phi(x) \Delta_g \psi(x) dx  \geq \int_X \Phi(x) \left( (2d-1+\epsilon)^2 -(2d-1)(2d-1+\epsilon)  \coth r(x) \right) \psi(x) dx. 
\end{align*}

Hence 
\begin{align*}
\Delta_g \Phi(x) \geq \Phi(x) \left( (2d-1+\epsilon)^2 -(2d-1)(2d-1+\epsilon)  \coth r(x) \right)
\end{align*}
in the sense of distributions on $X \setminus \{x_0\}$.
\end{proof}

%
%
%
%
%

\begin{proof}[Proof of Proposition~\ref{prop:hopf}]
Fix $r_0 > 0$ such that 
\begin{align*}
x  \rightarrow  \exp \Big( -(2d-1+\epsilon) d_g(x,x_0) \Big)
\end{align*}
is subharmonic on $X \setminus B_g(x_0,r_0)$. Since $\varphi < 0$, there exists $C >0$ such that 
\begin{align*}
\varphi(x) \leq -C\exp \Big( -(2d-1+\epsilon) d_g(x,x_0) \Big)
\end{align*}
for all $x \in B_g(x_0,r_0)$. Then consider 
\begin{align*}
f(x) = \varphi(x) + C\exp \Big( -(2d-1+\epsilon) d_g(x,x_0) \Big).
\end{align*}
Then $f$ is subharmonic on $X \setminus B_g(x_0,r)$. Fix $R > r_0$ and let 
\begin{align*}
A_{R} = B_g(x_0,R) \setminus B_g(x_0,r_0)
\end{align*}
Then $f(x) \leq 0$ on $\partial B_g(x_0,r_0)$ and 
\begin{align*}
f(x) \leq C\exp \Big( -(2d-1+\epsilon) R \Big)
\end{align*}
on $\partial B_g(x_0,R)$. So by the maximum principle 
\begin{align*}
f(x) \leq C\exp \Big( -(2d-1+\epsilon) R \Big)
\end{align*}
on $A_{R}$. Then sending $R \rightarrow 0$ shows that 
\begin{align*}
f(x) \leq 0
\end{align*}
on $X \setminus B_g(x_0,r_0)$. So 
\begin{align*}
\varphi(x) \leq -C\exp \Big( -(2d-1+\epsilon) d_g(x,x_0) \Big)
\end{align*}
for all $x \in X$. 
\end{proof}

\section{Proof of Theorem~\ref{thm:main}}

Suppose $\Omega \subset \Cb^d$ is a bounded pseudoconvex domain with $\alpha(\Omega) > 0$, $g$ is a complete K{\"a}hler metric on $\Omega$ with ${ \rm Ric}_g \geq -(2d-1)$, $z_0 \in \Omega$, and $\epsilon > 0$. 

Fix $\epsilon_1 > 0$ and a hyperconvexity exponent $\tau \in (0,1)$ such that 
\begin{align*}
\frac{\tau}{2d-1+\epsilon_1} \geq \frac{\alpha(\Omega)}{2d-1} - \epsilon. 
\end{align*}
Then there exists a continuous plurisubharmonic function $\psi :\Omega \rightarrow (-\infty,0)$ and $a > 1$ such that 
\begin{align*}
-\psi(z) \leq a \delta_\Omega(z)^{\tau}
\end{align*}
for all $z \in \Omega$.

Since $\psi$ is plurisubharmonic and $g$ is K\"ahler, $\psi$ is subharmonic on $(\Omega, g)$. So by Proposition~\ref{prop:hopf} there exists $C_0 > 0$ such that 
\begin{align*}
\psi(z) \leq -C_0\exp \Big( -(2d-1+\epsilon_1) d_g(x,x_0) \Big)
\end{align*}
for all $z \in \Omega$. Then 
\begin{align*}
-a\delta_\Omega(z)^{\tau} \leq -C_0\exp \Big( -(2d-1+\epsilon_1) d_g(x,x_0) \Big)
\end{align*}
and so there exists $C_1 \in \Rb$ such that 
\begin{align*}
C_1 + \left( \frac{\tau}{2d-1+\epsilon_1} \right) \log \frac{1}{\delta_\Omega(z)} \leq d_g(z,z_0)
\end{align*}
for all $z \in \Omega$. Since the set $\{ z \in \Omega : \delta_\Omega(z) \geq 1\}$ is compact and
\begin{align*}
\frac{\tau}{2d-1+\epsilon_1} \geq \frac{\alpha(\Omega)}{2d-1} - \epsilon,
\end{align*}
 there exists $C \in \Rb$ such that  
\begin{align*}
C + \Big( \frac{\alpha(\Omega)}{2d-1}-\epsilon \Big) \log \frac{1}{\delta_\Omega(z)} \leq d_g(z,z_0)
\end{align*}
for all $z \in \Omega$.

\bibliographystyle{alpha}
\bibliography{complex_kob}

\end{document}